\documentclass[12pt]{amsart}        
\usepackage{latexsym,mathtools}    
\usepackage{mathrsfs}
\usepackage{epsfig,setspace} 
\usepackage{a4}
\usepackage{comment}
\usepackage{amssymb}
\usepackage{amsthm}
\usepackage{amsbsy} 
\usepackage{upgreek}
\usepackage{relsize}
\usepackage{tikz-cd}
\usepackage{ytableau}
\usepackage{amsfonts,amssymb,latexsym,amscd,amsmath,euscript,enumerate,verbatim,calc}
\allowdisplaybreaks
\usepackage{url}
\usepackage{hhline}
\usepackage{pifont}
\usepackage{relsize}
\usepackage[toc,page]{appendix}
\usepackage[backend=biber,maxnames=10]{biblatex}  

\DeclareFieldFormat{doi}{\href{https://dx.doi.org/#1}{\textsc{doi}}}
\DeclareFieldFormat{url}{\href{#1}{\textsc{url}}}
\renewbibmacro*{doi+eprint+url}{
  \printfield{doi}
  \newunit\newblock
  \iftoggle{bbx:eprint}{
    \usebibmacro{eprint}
  }{}
  \newunit\newblock
  \iffieldundef{doi}{
    \usebibmacro{url+urldate}}
  {}
}
\renewbibmacro{in:}{}
\AtEveryBibitem{
  \ifentrytype{book}
  {\clearfield{pages}}
}

\usepackage[colorlinks=true,linkcolor=blue,citecolor=magenta]{hyperref} 

\theoremstyle{definition}

\newtheorem*{theorem*}{Theorem}
\newtheorem{theorem}{Theorem}[section]   
\newtheorem{lemma}[theorem]{Lemma}
\newtheorem{proposition}[theorem]{Proposition}

\newtheorem{problem}[theorem]{Problem}
 
\newtheorem{corollary}[theorem]{Corollary} 
\newtheorem{definition}[theorem]{Definition}

\newtheorem{remark}[theorem]{Remark} 
\newtheorem{example}[theorem]{Example}

\def\F{{\mathbb F}}
\def\ZZ{{\mathbb Z}}
\def\QQ{{\mathbb Q}}
\def\Fq{{\mathbb F}_q}
\def\MM{{\mathrm M}}
\def\mm{{\mathsf m}}
\def\he{\mathscr{H}}
\makeatletter
\@namedef{subjclassname@2020}{\textup{2020} Mathematics Subject Classification} 
\makeatother

\title[Subspace Profiles, $q$-Whittaker functions and Krylov methods]{Subspace profiles over finite fields and\\ $q$-Whittaker expansions of symmetric functions}       
\author{Samrith Ram} 
\address{Indraprastha Institute of Information Technology Delhi, New Delhi, India.}
\email{samrithram@gmail.com}

\begin{document} 
\begin{abstract}
 Bender, Coley, Robbins, and Rumsey posed the problem of counting the number of subspaces which have a given profile with respect to a linear endomorphism defined on a finite vector space. We settle this problem in full generality by giving an explicit counting formula involving symmetric functions. This formula can be expressed compactly in terms of a Hall scalar product involving dual $q$-Whittaker functions and another symmetric function that is determined by conjugacy class invariants of the endomorphism. As corollaries, we obtain new combinatorial interpretations for the coefficients in the $q$-Whittaker expansions of several symmetric functions. These include the power sum, complete homogeneous, products of modified Hall-Littlewood polynomials, and certain products of $q$-Whittaker functions. These results are used to derive a formula for the number of anti-invariant subspaces (as defined by Barría and Halmos) with respect to an arbitrary operator. We also give an application to an open problem in Krylov subspace theory.    
\end{abstract}
\maketitle
\tableofcontents

\section{Introduction}\label{sec:introduction}

Let $\Fq$ denote the finite field with $q$ elements where $q$ is a prime power. For each positive integer $n$, write $\MM_n(\Fq)$ for the algebra of $n\times n$ matrices over $\Fq$. 
\begin{definition}
Given a matrix $\Delta \in \MM_n(\Fq)$, a subspace $W$ of $\Fq^n$ has $\Delta$-profile $\mu=(\mu_1,\mu_2,\ldots)$ if  
\begin{align*}
  \dim (W+\Delta W+\cdots +\Delta^{j-1}W)=\mu_1+\mu_2+\cdots+\mu_{j} \mbox{ for }j\geq 1.
\end{align*}
  
\end{definition}

Let $\sigma(\mu,\Delta)$ denote the number of subspaces with $\Delta$-profile $\mu$. The $\Delta$-profile of a subspace was referred to as `dimension sequence' by Bender, Coley, Robbins and Rumsey \cite[p. 2]{MR1141317} who proposed the following combinatorial problem.
\begin{problem}\label{prob:main}
  Given $\mu$ and $\Delta$, determine $\sigma(\mu,\Delta)$.
\end{problem}
They showed, by a beautiful probabilistic argument involving Möbius inversion on the lattice of subspaces, that the quantities $\sigma(\mu,\Delta)$ for $\mu$ varying satisfy a system of equations (see Theorem \ref{thm:bendereqns}). They solved these equations in two special cases to obtain elegant product formulas. If $\Delta$ is regular nilpotent (nilpotent with one-dimensional null space), then
\begin{equation}\label{eq:regnilprofiles}
  \sigma(\mu,\Delta) = \prod_{i\geq 2}q^{\mu_i^2}{\mu_{i-1} \brack \mu_i}_q.
\end{equation}
 When $\Delta$ is  simple (has irreducible characteristic polynomial),
  \begin{equation}\label{eq:simpleprofiles}
  \sigma(\mu,\Delta) = \frac{q^n-1}{q^{\mu_1}-1}\prod_{i\geq 2}q^{\mu_i^2-\mu_i}{\mu_{i-1} \brack \mu_i}_q.
\end{equation}
They remarked that these formulas do not appear to have simple counting proofs. 

 A few years later, Niederreiter \cite[p. 11]{MR1334623} proposed the following problem in connection with his multiple recursive matrix method for pseudorandom number generation.
\begin{problem}
  Let $m,\ell$ be positive integers. Given an element $\alpha\in \F_{q^{m\ell}}$ such that $\F_{q^{m\ell}}=\Fq(\alpha)$, determine the number of $m$-dimensional $\Fq$-subspaces $W$ of $\F_{q^{m\ell}}$  such that
  \begin{align*}
    \F_{q^{m\ell}}=W\oplus \alpha W\oplus \cdots \oplus \alpha^{\ell-1}W.
  \end{align*}
\end{problem} 
It is not difficult to see that Niederreiter's problem is equivalent to determining $\sigma(\mu,\Delta)$ in the case where $\mu$ is a partition of $m\ell$ with all parts equal to $m$ and $\Delta$ is an $m\ell\times m\ell$ matrix with irreducible characteristic polynomial. However, Niederreiter does not seem to have been aware of the results of the authors of \cite{MR1141317} and the same is true for several subsequent papers which discuss the problem of Niederreiter. Niederreiter's problem has interesting connections with finite projective geometry and group theory  \cite{split}. Several special instances of Problem~\ref{prob:main} have been solved in the literature \cite{MR2831705,split,MR4448290,MR4263652,MR3093853,MR4682040,MR4797454,MR4555237,ram2023diagonal}. The consideration of subspace profiles when $\Delta$ is a regular diagonal operator has recently led to a new proof of the Touchard-Riordan formula concerned with crossings of chord diagrams \cite{MR4555237}. In the special case where $\Delta$ is a diagonal matrix, the quantity $\sigma(\mu,\Delta)$ can be expressed as a sum over semistandard tableaux. The study of subspace profiles in this case has interesting connections with multiset Mahonian statistics, $q$-Whittaker symmetric functions and the $q$-rook theory of Garsia and Remmel. We refer to \cite{ram2023diagonal} for these results on diagonal operators. 
 
In this paper we solve Problem \ref{prob:main} in full generality by giving an explicit formula for $\sigma(\mu,\Delta)$ for arbitrary $\mu$ and $\Delta$. The fact that the theory of symmetric functions can be leveraged to answer the counting problem is a very recent development. In addition to the results of Bender, Coley, Robbins and Rumsey, our primary tools are ideas from the theory of symmetric functions and some previous results for diagonal operators. We show that Problem \ref{prob:main} admits a compact solution involving $q$-Whittaker functions, a class of symmetric functions that occur as specializations of Macdonald polynomials. The $q$-Whittaker functions occur as joint eigenfunctions of $q$-deformed Toda chain Hamiltonians with support in the positive Weyl chamber (see Etingof \cite{MR1729357} or Ruijsenaars \cite{MR1090424}). The following theorem (appearing later as Theorem \ref{thm:allprofiles}) is our main result.   
\begin{theorem}\label{thm:main}
For each partition $\mu$, 
\begin{align*}
 \sigma(\mu,\Delta)=(-1)^{\sum_{j\geq 2}\mu_j} q^{\sum_{j\geq 2}{\mu_j \choose 2}}\langle  F_\Delta(x),\widetilde{W}_{\mu}(x;q)\, h_{n-|\mu|} \rangle,
\end{align*}
for each prime power $q$ and each matrix $\Delta\in \MM_n(\Fq)$.
\end{theorem} 
Here $\langle \cdot,\cdot \rangle$ denotes the Hall scalar product while $h_\lambda$ and $\widetilde{W}_\lambda$ denote the complete homogeneous and dual (with respect to the Hall scalar product) $q$-Whittaker symmetric functions  respectively. In addition, $F_\Delta(x)$ is a symmetric function depending on conjugacy class invariants of the matrix $\Delta$ that can be expressed in terms of plethystic substitutions involving modified Hall-Littlewood polynomials (see Proposition \ref{prop:ifgf}). The ideas used in the proof of Theorem \ref{thm:main} are elementary and combinatorial; the arguments do not use representation theoretic or geometric techniques. 

We show that several symmetric functions such as the power sum symmetric functions, the complete homogeneous symmetric functions, products of modified Hall-Littlewood polynomials, and certain products of $q$-Whittaker functions are of the form $F_\Delta$ for suitably chosen $\Delta$ (see Section \ref{sec:fdelta}). In the case where $\mu$ is a partition of $n$, the theorem above entails new combinatorial interpretations of the coefficients in the $q$-Whittaker expansion of each of these symmetric functions (Corollary \ref{cor:comb}). 

In Section \ref{sec:partialprofiles}, we use our results to derive an explicit formula for the number of anti-invariant subspaces with respect to a linear operator, a notion that goes back to Barría and Halmos \cite{MR748946}. An application of our results to Krylov subspace theory is considered in Section \ref{sec:krylov}. Given positive integers $k$ and $\ell$, we assign a probability $\psi_{k,\ell}(\Delta)$ to each matrix $\Delta\in \MM_n(\Fq)$ (see Equation \eqref{eq:defpsi}). Obtaining bounds on the probability $\psi_{k,\ell}(\Delta)$ is a difficult and important problem (Brent, Gao and Lauder \cite[p.\ 277]{MR1982139}) whose resolution is critical for understanding and evaluating the efficiency of several algorithms that rely on Krylov subspace methods. Examples of such algorithms include the Number Field Sieve, used in integer factorization, and the block Wiedemann algorithm of Coppersmith \cite{MR1192970}. In general, such algorithms have applications in many mathematical areas, including quadrature methods, the analytic theory of continued fractions, expansions of infinite series, orthogonalization algorithms, and the mathematical underpinnings of quantum mechanics (Liesen and Strakoš \cite[p.\ 8]{MR3024841}). In Theorem \ref{thm:krylov}, we derive an explicit formula for the probability $\psi_{k,\ell}(\Delta)$ in terms of dual $q$-Whittaker functions.

This article is organized as follows. In Section \ref{sec:fdelta}, we begin with a brief overview of the algebra of symmetric functions and its various bases. We show how to associate a symmetric function $F_\Delta$ with each matrix $\Delta$ and give several examples. We also prove the existence of a formula which expresses $\sigma(\mu,\Delta)$ in terms of $F_\Delta$. In Section \ref{sec:diagonal}, we consider diagonal operators and $q$-Whittaker functions which will later play an integral part in the derivation of a formula for $\sigma(\mu,\Delta)$. Section~\ref{sec:fullprofiles} is concerned with a formula for $\sigma(\mu,\Delta)$ when $\mu$ is a partition of the ambient vector space dimension. This formula is derived by setting up a system of equations and using a linear algebraic approach. We also give combinatorial finite-field interpretations for coefficients in the $q$-Whittaker expansions of various symmetric functions. Finally, in Section \ref{sec:allprofiles}, we use the equations of Bender, Coley, Robbins and Rumsey to determine $\sigma(\mu,\Delta)$ for arbitrary $\mu$ and $\Delta$. Applications to anti-invariant subspaces and Krylov subspace methods are considered in Sections \ref{sec:partialprofiles} and \ref{sec:krylov} respectively.   

\section{Generating function for flags of invariant subspaces}\label{sec:fdelta}

We begin with a brief overview of symmetric functions, focusing on some relevant facts; the primary  references are Macdonald~\cite{MR1354144} and Stanley \cite[Chap. 7]{MR1676282}. A weak composition of an integer $n$ is a sequence $\alpha=(\alpha_1,\alpha_2,\ldots)$ of nonnegative integers with sum $n$. A partition of $n$ is a weak composition of $n$ in which the sequence of integers is weakly decreasing. If $\lambda$ is a partition of $n$, we write $\lambda \vdash n$. Nonzero terms in the sequence $\lambda$ are called parts of $\lambda$ and the number of parts of $\lambda$ is denoted $\ell(\lambda)$. It is customary to ignore trailing zeroes in weak compositions and partitions. For instance, the weak compositions $(3,1,2,0,1,0,0,\ldots)$, $(3,1,2,0,1,0)$ and $(3,1,2,0,1)$ are all considered equivalent.  For a weak composition $\alpha$, write $|\alpha|$ for the sum $\sum_{i\geq 1}\alpha_i$. 

Let $\QQ(t)$ denote the field of rational functions in an indeterminate $t$. Denote by $\Lambda_{\QQ(t)}$ the algebra of formal symmetric functions in infinitely many variables $x=(x_1,x_2,\ldots)$ with coefficients in $\QQ(t)$. The algebra $\Lambda_{\QQ(t)}$ admits several natural bases indexed by integer partitions: the monomial symmetric functions $m_\lambda$, the elementary symmetric functions $e_\lambda$, the power sum symmetric functions $p_\lambda$, the complete homogeneous symmetric functions $h_\lambda$ and the Schur functions $s_\lambda$. The ring of symmetric functions is also equipped with an involutory automorphism $\omega$ which satisfies
\begin{align*}
\omega e_\lambda= h_\lambda;\quad \omega h_\lambda= e_\lambda;\quad \omega s_\lambda= s_{\lambda'}.
\end{align*}
Here, and throughout this paper, $\lambda'$ denotes the partition conjugate to $\lambda$.

The $q$-Whittaker functions $W_\lambda(x;t)$ and Hall-Littlewood polynomials $P_\lambda(x;t)$ are two more bases of $\Lambda_{\QQ(t)}$. Both occur as specializations of a more general class of two-parameter symmetric functions, the Macdonald polynomials.

The ring of symmetric functions is endowed with the Hall scalar product $\langle \cdot , \cdot\rangle$ with respect to which the bases $m_\lambda$ and $h_\lambda$ are dual. The dual basis of the Hall-Littlewood polynomials $P_\lambda(x;t)$ with respect to the Hall scalar product consists of the transformed Hall-Littlewood polynomials $H_{\lambda}(x;t)$. They are related to the $q$-Whittaker functions (Bergeron \cite[p. 5]{bergeron2020survey}) by
\begin{align}
\omega H_{\lambda'}(x;t)=W_{\lambda}(x;t).\label{eq:hisomegaw}
\end{align}

The modified Hall-Littlewood polynomials are indexed by integer partitions and are defined by $  \tilde{H}_{\lambda}(x;t)=t^{n(\lambda)}  H_{\lambda}(x;t^{-1}),$ where $n(\lambda)=\sum_{i\geq 1}(i-1)\lambda_i$. The modified Hall-Littlewood polynomial also satisfies
\begin{align*}
  \tilde{H}_\lambda(x;t)=\sum_{\mu}\tilde{K}_{\mu\lambda}(t)s_\mu,
\end{align*}
where $s_\mu$ denotes a Schur function and $\tilde{K}_{\mu\lambda}(t)$ is a modified Kostka-Foulkes polynomial, defined as the generating polynomial of the cocharge statistic on semistandard tableaux of shape $\mu$ and content $\lambda$:
\begin{align*}
  \tilde{K}_{\mu\lambda}(t)=\sum_{\mathcal{T} \in {\rm SSYT}(\mu,\lambda)}t^{\rm cocharge{(\mathcal{T})}}.
\end{align*}
For prime power $q$, the specialization $\tilde{K}_{\lambda\mu}(q)$ coincides with the value taken by the character of the irreducible unipotent ${\rm GL}_n(\Fq)$-representation indexed by $\lambda$ on the unipotent conjugacy class with Jordan form partition $\mu$ (Lusztig \cite[Eq.~2.2]{MR641425}).

For a formal power series $F$, denote by $p_k[F]$ the series obtained by substituting each indeterminate appearing in $F$ by its $k$th power. For an arbitrary symmetric function $g\in \Lambda_{\QQ(t)}$, the \emph{plethystic substitution} $g[F]$ is the series obtained by writing $g$ as a polynomial in the power sum symmetric functions $p_r$ and then substituting $p_r[F]$ for each $p_r$ appearing in $g$. The reader is referred to Haglund \cite[p. 19]{MR2371044} or Haiman \cite[p. 13]{MR2051783} for further details on plethystic notation.

 For each nonnegative integer $n$, define the $q$-analogs 
$[n]_q:=1+q+\cdots+q^{n-1}\mbox{ and }[n]_q!:=[1]_q[2]_q\cdots [n]_q.$ For a weak composition $\alpha$ of $n$, the $q$-multinomial coefficient is defined by
\begin{align*}
  {n \brack \alpha}_q:=\frac{[n]_q!}{\prod_{i\geq 1}[\alpha_i]_q!}.
\end{align*}

Our objective, in this section, is to show that $\sigma(\mu,\Delta)$ can be expressed in terms of suitably defined flags of invariant subspaces of $\Delta$. We begin with a definition.
\begin{definition}
Given a matrix $\Delta\in \MM_n(\Fq)$ and a weak composition $\alpha=(\alpha_1,\ldots,\alpha_\ell)$ of $n$, let $X_{\alpha}(\Delta)$ denote the number of flags $  (0)=W_0\subseteq W_1\subseteq \cdots \subseteq W_\ell=\Fq^n$ of $\Delta$-invariant subspaces satisfying $\dim W_i/W_{i-1}=\alpha_i$ for $1\leq i\leq r$.
\end{definition}
\begin{example}\label{eg:scalar}
If $\Delta=cI$ where $I$ denotes the $n\times n$ identity matrix and $c\in \Fq$, then $X_\alpha(\Delta)={n \brack \alpha}_q,$ a $q$-multinomial coefficient.
\end{example}

Let $\mathcal{L}(\Delta)$ denote the lattice of $\Delta$-invariant subspaces of $\Fq^n$. It is well known that the lattice $\mathcal{L}(\Delta)$ is self-dual \cite{MR213378}. It follows easily from this fact that every interval in $\mathcal{L}(\Delta)$ is self-dual. The next proposition shows that $X_\alpha(\Delta)$ depends only on the integer partition obtained by sorting the coordinates of $\alpha$ in weakly decreasing order.

\begin{proposition}\label{prop:sortparts}
If $\alpha$ and $\alpha'$ are weak compositions such that the coordinates of $\alpha'$ are a permutation of the coordinates of $\alpha$, then $X_\alpha(\Delta)=X_{\alpha'}(\Delta).$ 
\end{proposition}
\begin{proof}
We may assume that the permutation taking $\alpha$ to $\alpha'$ involves only finitely many coordinates.  Since each such permutation is obtained by a series of transpositions of adjacent elements, it suffices to prove the proposition when $\alpha'$ is obtained from $\alpha$ by swapping two adjacent coordinates, say $\alpha_i$ and $\alpha_{i+1}$. Let $\mathcal{F}_\alpha(\Delta)$ denote the collection of all flags counted by $X_\alpha(\Delta)$. We will construct an explicit bijection between $\mathcal{F}_\alpha(\Delta)$ and $\mathcal{F}_{\alpha'}(\Delta)$. Consider a flag
  $$(0)= W_0\subseteq \cdots \subseteq W_{i-1}\subseteq W_i\subseteq W_{i+1}\subseteq\cdots \subseteq W_\ell=\Fq^n$$
  in $\mathcal{F}_\alpha(\Delta)$. The interval $[W_{i-1},W_{i+1}]\subseteq \mathcal{L}(\Delta)$ is self-dual and $\dim W_{i+1}/W_{i-1}=\alpha_i+\alpha_{i+1}$. Self-duality implies the existence of a bijection $\phi:[W_{i-1},W_{i+1}]\to [W_{i-1},W_{i+1}]$  which takes each invariant subspace $U$ with $\dim U/W_{i-1}=\alpha_i$ to a unique invariant subspace $U'$ satisfying $\dim U'/W_{i-1}=\alpha_{i+1}$. The map
  \begin{align*}
    (W_0,\ldots, W_{i-1}, W_i, W_{i+1},\ldots, W_\ell)\mapsto     (W_0,\ldots, W_{i-1}, \phi(W_i), W_{i+1},\ldots, W_\ell)
  \end{align*}
  is easily seen to be a bijection between $\mathcal{F}_\alpha(\Delta)$ and $\mathcal{F}_{\alpha'}(\Delta)$.
\end{proof}
To study flags of invariant subspaces it will be convenient to introduce a generating function in infinitely many variables $x=(x_1,x_2,\ldots)$.
\begin{definition}\label{def:ifgf}
For $\Delta\in \MM_n(\Fq)$, the \emph{invariant flag generating function} $F_{\Delta}(x)$ is defined by 
 \begin{align*}
   F_{\Delta}(x):=\sum_{\alpha }X_{\alpha}(\Delta)x^{\alpha},
 \end{align*}
where the sum is taken over all inequivalent weak compositions $\alpha$ of $n$ and $x^{\alpha}$ denotes the product $x_1^{\alpha_1}x_2^{\alpha_2}\cdots$.
\end{definition}

It is important to take the sum over inequivalent weak compositions of $n$ since contributions from compositions obtained from $\alpha$ by padding zeroes would make the coefficient of $x^\alpha$ infinite. The following result is an easy consequence of Proposition~\ref{prop:sortparts} and the definition of the monomial symmetric function $m_\lambda$.

\begin{proposition}
For each matrix $\Delta \in \MM_n(\Fq)$,
\begin{align*}
  F_\Delta(x)=\sum_{\lambda}X_{\lambda}(\Delta)m_{\lambda},
 \end{align*}
where the sum is taken over all partitions $\lambda$ of $n$.
\end{proposition}

\begin{example}\label{eg:simple}
  If $\Delta\in \MM_n(\Fq)$ is \emph{simple} (it has irreducible characteristic polynomial), then the only $\Delta$-invariant subspaces are the zero subspace and $\Fq^n$. Therefore $F_\Delta(x)=\sum_{i\geq 1}x_i^n=p_n$, the power sum symmetric function.
\end{example}
\begin{example}\label{eg:regularnilpotent}
  If $\Delta\in \MM_n(\Fq)$ is regular nilpotent (nilpotent with one-dimensional null space), then there is precisely one $\Delta$-invariant subspace of dimension $k$ for each $0\leq k\leq n$. In this case $F_\Delta(x)=\sum_{\alpha}x^\alpha=h_n$, the complete homogeneous symmetric function.
\end{example}

The following lemma follows easily from Kirillov \cite[Eq. 0.4]{MR1768934} where it is stated for unipotent operators.
\begin{lemma}\label{lem:nilflags} 
The invariant flag generating function of a nilpotent matrix over $\Fq$ with Jordan form partition $\lambda$ is the modified Hall-Littlewood polynomial $\tilde{H}_{\lambda}(x;q)$.
\end{lemma}  

Our objective now is to determine the invariant flag generating function for an arbitrary matrix $\Delta\in \MM_n(\Fq)$. Recall that the action of $\Delta$ on $\Fq^n$ defines an $\Fq[t]$-module structure on $\Fq^n$ in which the action of $t$ on a column vector $v$ is given by $t\cdot v=\Delta v$. By the structure theorem for finitely generated modules over a principal ideal domain, this module is isomorphic to a direct sum
\begin{align}
 \bigoplus_{i=1}^k \bigoplus_{j=1}^{\ell_i}\frac{\Fq[t]}{(g_i^{\lambda_{i,j}})},\label{eq:isom}
\end{align}
where $g_i(t)\in \Fq[t]$ are distinct monic irreducible polynomials and the sequence $\lambda^i=(\lambda_{i,1},\lambda_{i,2},\ldots,\lambda_{i,\ell_i})$ is an integer partition for each $1\leq i\leq k$.  Let $d_i$ denote the degree of $g_i$ for $1\leq i\leq  k$.
\begin{definition}
With $d_i$ and $\lambda_i$ as above, the \emph{similarity class type} of the matrix $\Delta$ is the multiset $\tau=\{(d_1,\lambda^1),(d_2,\lambda^2),\ldots,(d_k,\lambda^k)\}$. The \emph{size} of $\tau$ is the integer $\sum_{i=1}^k d_i|\lambda^i|$.
\end{definition}
The notion of similarity class type can be traced back to the work of Green~\cite{MR72878} who studied the characters of the finite general linear groups. Considering similarity class types allows for a $q$-independent classification of conjugacy classes in these groups.

\begin{example} Consider $n\times n$ matrices over $\Fq$.
  \begin{enumerate}
  \item A simple matrix has type $\{(n,(1))\}$.
  \item A scalar multiple of the identity has type $\{(1,(1^n))\}.$  
  \item A regular nilpotent matrix has type $\{(1,(n))\}$.
  \end{enumerate}
\end{example}
From a combinatorial viewpoint one advantage of working with types is that several combinatorial invariants of a matrix often depend only on its type. We will soon see that this viewpoint is particularly convenient in the context of invariant flag generating functions. We have the following explicit formula for $F_\Delta(x)$.
\begin{proposition}\label{prop:ifgf}
  If $\Delta \in \MM_n(\Fq)$ is a matrix of similarity class type  $\tau=\{(d_i,\lambda^i)\}_{1\leq i\leq k}$, then
  \begin{align}
  F_{\Delta}(x)=\prod_{i=1}^k\tilde{H}_{\lambda^i}(x_1^{d_i},x_2^{d_i},\ldots;q^{d_i})=\prod_{i=1}^k p_{d_i}[\tilde{H}_{\lambda^i}(x;t)]_{|t=q}, \label{eq:ifgf}
\end{align}
where $p_r$ denotes the power sum symmetric function. Here, the plethystic substitution $p_{d_i}[\tilde{H}_{\lambda^i}(x;t)]$ is performed before evaluating at $t=q$.
\end{proposition}

\begin{proof}
Since $\mathcal{L}(\Delta)$ is the product of the invariant subspace lattices corresponding to its primary parts (Brickman and Fillmore \cite[Thm. 1]{MR213378}), it can be seen that  
\begin{align}
  F_\Delta=\prod_{i=1}^k F_{\Delta_i},\label{eq:primdecomp}
\end{align}
where $\Delta_i$ denotes the primary part of $\Delta$ with similarity class type $\{(d_i,\lambda^i)\}$ for $1\leq i\leq k$. If $g(t)\in \Fq[t]$ is an irreducible polynomial of degree $d$, then $\Fq[t]/(g^r)$ is isomorphic to $\mathbb{F}_{q^d}[u]/(u^r)$. In view of Equation \eqref{eq:isom} and Lemma \ref{lem:nilflags}, this yields
\begin{align*}
  F_{\Delta_i}= \tilde{H}_{\lambda^i}(x_1^{d_i},x_2^{d_i},\ldots;q^{d_i})=p_{d_i}[\tilde{H}_{\lambda^i}(x;t)]_{|t=q},
\end{align*}
and the result follows.
\end{proof}
The expression for $F_\Delta$ in Proposition~\ref{prop:ifgf} is a product of symmetric functions where the parameter $t$ is specialized to a prime power $q$. Rather than specialize the parameter, one can work directly with the parametric versions $\tilde{H}_\lambda(x;t)$ by considering similarity class types instead of matrices. For a similarity class type $\tau=\{(d_1,\lambda^1),(d_2,\lambda^2),\ldots,(d_k,\lambda^k)\},$  define
\begin{align}
  F_{\tau}(x;t):=\prod_{i=1}^k p_{d_i}[\tilde{H}_{\lambda^i}(x;t)].\label{eq:ftau}
\end{align}
 Given a prime power $q$ and a matrix $\Delta\in \MM_n(\mathbb{F}_{q})$ with similarity class type $\tau$, it is clear that $F_\Delta(x)=F_{\tau}(x;q)$. With this approach statements about matrices can be restated in terms of similarity class types. For instance, instead of the quantity $X_\lambda(\Delta)$, one considers $X_\lambda(\tau)$ defined as the coefficient of the monomial $m_\lambda$ in $F_\tau(x;t)$ as in Equation \eqref{eq:ftau}. Although our statements will often involve matrices, it is useful to keep in mind the corresponding statements for similarity class types and we will give such examples.
\begin{remark}
For arbitrary $\tau$ and $q$, there may be no matrix of similarity class type $\tau$ over $\Fq$. However, for a fixed $\tau$, a matrix of type $\tau$ over $\Fq$ always exists for sufficiently large prime powers $q$.  
\end{remark}

We now give examples of various symmetric functions which arise as $F_\Delta$ for some $\Delta$. In Section \ref{sec:fullprofiles} we will see that a combinatorial interpretation can be given for the coefficients in the $q$-Whittaker expansions of each of these functions. The next two examples show that the power sum and homogeneous symmetric functions arise as invariant flag generating functions.
\begin{example}\label{eg:ifgfispower}
    A matrix is \emph{regular semisimple} if its characteristic polynomial is a product of distinct irreducible polynomials $g_i(t)(1\leq i\leq k)$ over $\Fq$. Let $\lambda=(\lambda_1,\ldots,\lambda_k)$ denote the integer partition obtained by arranging the degrees of the $g_i$ in weakly decreasing order. The invariant flag generating function is given by $\prod_{i=1}^k p_{\lambda_i}=p_\lambda$. 
\end{example}
\begin{example}\label{eg:regularsplit}
    A matrix is \emph{regular split} if its minimal polynomial is equal to its characteristic polynomial and each is a product of linear factors: $\prod_{i=1}^k (x-a_i)^{\lambda_i}$. In this case the invariant flag generating function is $\prod_{i=1}^k h_{\lambda_i}=h_\lambda.$
\end{example}
The following is a simultaneous generalization of Examples \ref{eg:regularnilpotent} and \ref{eg:regularsplit}.
\begin{example}
  A matrix is \emph{triangulable} if it is similar to an upper triangular matrix. Triangulability is characterized by the minimal polynomial being a product of linear factors \cite[Sec. 6.4]{MR0276251} over $\Fq$. For such a matrix $\Delta$, we have $d_i=1$ in the decomposition \eqref{eq:isom} above for $1\leq i\leq k$. Since $p_1[f]=f$ for any symmetric function $f$, it follows from Proposition \ref{prop:ifgf} that 
  \begin{align*}
    F_\Delta=\prod_{i\geq 1} \tilde{H}_{\lambda^i}(x;q).
  \end{align*}
\end{example}
The next example will prove useful later on in the proof of Theorem \ref{thm:fullprofiles}.
\begin{example}\label{eg:diagonalifgf}
Let $\Delta \in \MM_n(\Fq)$ be a diagonalizable matrix with characteristic polynomial $\prod_{i=1}^k (x-a_i)^{\nu_i}$ for some partition $\nu$ of $n$ and distinct elements $a_i\in \Fq(1\leq i\leq k).$ Since $\Delta$ acts by  a scalar multiple of the identity on each eigenspace, it follows from Example \ref{eg:scalar}, that $F_\Delta=\prod_{i\geq 1}F_{\nu_i}$ where, for each positive integer $m$,
  \begin{align*}
    F_m:=\sum_{\lambda \vdash m}{m \brack \lambda}_qm_{\lambda}=W_{(m)}(x;q),
  \end{align*}
  a $q$-Whittaker function (see Macdonald \cite[p. 323, Eq. 4.9]{MR1354144} and \cite[p. 314, Eg. 1]{MR1354144} for the last equality above). Therefore $ F_\Delta=\prod_{i\geq 1} W_{(\nu_i)}(x;q).$
\end{example}

Our immediate goal is to show that the number of subspaces with a given $\Delta$-profile can be expressed in terms of $F_\Delta(x)$.
\begin{definition}
Let $r$ be a positive integer and suppose $\Delta \in \MM_n(\Fq)$. Given partitions $\lambda=(\lambda_1,\ldots,\lambda_r)$ and $\nu=(\nu_1,\ldots,\nu_r)$, denote by $\phi(\lambda,\nu;\Delta)$ the number of flags $W_1\supseteq W_2\supseteq \cdots \supseteq W_r$ of subspaces of $\Fq^n$ such that
\begin{align*}
  \dim W_i=\lambda_i &\qquad\mbox{ for } 1\leq i\leq r,\\
  \dim (W_i\cap \Delta^{-1}W_i)=\nu_i &\qquad\mbox{ for } 1\leq i\leq r,\\
  W_i \cap \Delta^{-1}W_i\supseteq W_{i+1} &\qquad\mbox{ for }1\leq i\leq r-1.  
\end{align*}
\end{definition}
Note that $\phi(\lambda,\lambda;\Delta)$ is just equal to the number of flags $W_1\supseteq W_2\supseteq \cdots \supseteq W_r$ of $\Delta$-invariant subspaces such that $\dim W_i=\lambda_i$ for $1\leq i\leq r$.

The following lemma shows that $\sigma(\mu,\Delta)$ can be expressed in the form $\phi(\lambda,\nu;\Delta)$ for suitably chosen $\lambda$ and $\nu$ \cite[Prop. 4.6]{pr}. 
\begin{lemma}
  Given a partition $\mu=(\mu_1,\ldots,\mu_k)$, let $m_i=\mu_1+\cdots+\mu_i$ for $1\leq i\leq k$. We have
  \begin{align*}
    \sigma(\mu,\Delta)=\phi(\lambda,\nu;\Delta),
  \end{align*}
  where $\lambda=(m_k,m_{k-1},\ldots,m_1)$ and $\nu=(m_k,m_{k-1}-\mu_k,m_{k-2}-\mu_{k-1},\ldots,m_1-\mu_2)$.
\end{lemma}
Chen and Tseng \cite[Lemma 2.7]{MR3093853} (see \cite[Lem. 2.7]{MR4349887} for a more general result)  proved that the quantities $\phi(\lambda,\nu;\Delta)$ satisfy a recursion in which the base cases are of the form $\phi(\eta,\eta;\Delta)$. For $\Delta\in \MM_n(\Fq)$, this recursion involves coefficients that are products of $q$-binomial coefficients that depend only on $\lambda$ and $\nu$. These observations, together with Proposition \ref{prop:sortparts}, imply the following result.

\begin{theorem}\label{thm:formofanswer}
For each partition $\mu$ and each partition $\lambda$ of $n$, there exist polynomials $g_{\mu\lambda}(t)\in \ZZ[t]$, such that  
\begin{align*}
  \sigma(\mu,\Delta)=\sum_{\lambda \vdash n}g_{\mu \lambda}(q)X_\lambda(\Delta),
\end{align*}
for every prime power $q$ and every matrix $\Delta \in \MM_n(\Fq)$.
\end{theorem}
\section{Diagonal operators and $q$-Whittaker functions}\label{sec:diagonal}
We wish to determine the polynomials $g_{\mu\lambda}(t)$ appearing in Theorem \ref{thm:formofanswer}. This is accomplished by setting up a system of linear equations in Section \ref{sec:fullprofiles}. Diagonalizable operators play a central role in this analysis and we begin by discussing some relevant results. 

\begin{definition}
    The type of  a diagonalizable matrix is the integer partition obtained by sorting the dimensions of its eigenspaces in weakly decreasing order.
\end{definition}
Note that a diagonalizable matrix of type $\nu$ over $\Fq$ exists if and only if the number of parts of $\nu$ is at most $q$. The following result was proved in \cite[Thm.~4.21]{ram2023diagonal}.
\begin{theorem}\label{thm:diagonal}
For each pair $\mu,\nu$ of integer partitions of $n$, there exist polynomials $b_{\mu\nu}(t)\in \ZZ[t]$ such that
  \begin{align*}
    \sigma(\mu,\Delta)=(q-1)^{\sum_{j\geq 2}\mu_j}q^{\sum_{j\geq 2}{\mu_j \choose 2}}b_{\mu\nu}(q),
  \end{align*}
for each prime power $q$ and each diagonal matrix $\Delta\in \MM_n(\Fq)$ of type $\nu$. 
\end{theorem}

 The polynomials $b_{\mu\nu}(t)$ satisfy the following recurrence relation \cite[Cor. 5.2]{ram2023diagonal} which was derived by considering the decomposition of the open Schubert cells in the Grassmanian according to subspace profiles with respect to a diagonal operator.
\begin{proposition}\label{prop:recforb} 
If $\mu,\nu$ are partitions of the same size with $\nu=(\nu_1,\ldots,\nu_k),$ then 
\begin{align*}
  b_{\mu\nu}(t)=\sum_{\substack{\rho:\mu/\rho\text{ is a horizontal}\\ \text{strip of size } \nu_k}}\theta_{\mu/\rho}(t)b_{\rho \tilde{\nu}}(t) , 
\end{align*}
where $\tilde{\nu}$ denotes the partition obtained by deleting the last part of $\nu$ and
\begin{align}
  \theta_{\mu/\rho}(t):=\frac{[|\mu|-|\rho|]_t!}{[\mu_1-\rho_1]_t!} \prod_{i\geq 1}{\rho_{i}-\rho_{i+1} \brack \mu_{i+1}-\rho_{i+1}}_t.\label{eq:skewfunc}
\end{align}
\end{proposition}
Note that $  \theta_{\mu/\rho}(t)=0$ unless the skew diagram $\mu/\rho$ is a horizontal strip (see Macdonald \cite[p. 5]{MR1354144} for the definition).
\begin{remark}\label{rem:bvanishes}
It follows from Proposition \ref{prop:recforb} that $b_{\mu\nu}(t)$ can be written as a sum over semistandard tableaux of shape $\mu$ and content $\nu$ \cite[Def. 3.7]{ram2023diagonal}. In particular, $b_{\mu\nu}(t)=0$ unless $\mu\geq \nu$ in the dominance order (see Stanley \cite[Sec. 7.2]{MR1676282} for the definition). It also follows from the form of $\theta_{\mu/\rho}(t)$ that $b_{\mu\nu}(t)$ is a polynomial in $t$ with nonnegative integer coefficients which is nonzero whenever $\mu\geq \nu$.
\end{remark}

The polynomials $b_{\mu\nu}(t)$ can also be obtained from a statistic on a suitably defined class of set partitions. The fact that these polynomials can be expressed in terms of both set partition statistics and semistandard tableaux has led to an elementary correspondence between these two classical combinatorial classes. This correspondence yields a way to associate a set partition statistic to each Mahonian statistic on multiset permutations. The polynomials $b_{\mu\nu}(t)$ also have several interesting specializations. In particular, when $\mu=(m,m)$ and $\nu=(1^{2m})$ the polynomial $b_{\mu\nu}(t)$ coincides with the Touchard-Riordan generating polynomial for chord diagrams by their number of crossings. For more on this topic and some connections with $q$-rook theory, the reader is referred to \cite{ram2023diagonal}. Here we discuss a connection between these polynomials and $q$-Whittaker functions.

 Though they were originally studied by Macdonald \cite{MR1354144}, the name $q$-Whittaker was coined by Gerasimov, Lebedev and Oblezin \cite{MR2575477}. As noted earlier, the $q$-Whittaker functions $W_\lambda(x;t)$ form a basis for $\Lambda_{\QQ(t)}$. They interpolate between the Schur functions and the elementary symmetric functions:  
\begin{align*}
  W_\lambda(x;0)=s_{\lambda}(x), \qquad W_\lambda(x;1)=e_{\lambda'}(x).
\end{align*}
They also expand positively in the Schur basis,
\begin{align}
  W_{\mu}=\sum_{\lambda}K_{\lambda'\mu'}(t)s_\lambda,\label{eq:whittoschur}
\end{align}
where $K_{\lambda\mu}(t)$ denotes a Kostka-Foulkes polynomial (Bergeron \cite[Eq. 3.8]{bergeron2020survey}). In a representation theoretic context, the $q$-Whittaker functions arise in the setting of the graded Frobenius characteristic of the cohomology ring of Springer fibers. We refer to the survey article of Bergeron \cite{bergeron2020survey} for more on $q$-Whittaker functions.

The polynomials $b_{\mu\nu}(t)$ are closely related to the coefficients in the monomial expansion of the $q$-Whittaker function $W_\lambda$. More precisely, we have the following proposition \cite[Thm. 5.3]{ram2023diagonal}.
\begin{proposition}\label{prop:bmunuascoeff}
  For partitions $\mu$ and $\nu$,  
\begin{align*}
  b_{\mu\nu}(t)=\frac{\prod_{i\geq 1}[\nu_i]_t!}{\prod_{i\geq 1}[\mu_i-\mu_{i+1}]_t!}\langle W_\mu,h_\nu\rangle.
\end{align*}
\end{proposition}

\begin{remark}
  The symmetric functions $\widetilde{W}_\lambda$ defined by
  \begin{align}
    \widetilde{W}_\lambda(x;t)=\frac{(1-t)^{-\lambda_1}}{\prod_{i\geq 1}[\lambda_i-\lambda_{i+1}]_t!}W_\lambda[X(1-t)] \label{eq:dualwhit}
  \end{align}
  are dual to the $q$-Whittaker functions with respect to the Hall scalar product (Bergeron \cite{bergeron2020survey}). We have the following dual of Equation \eqref{eq:hisomegaw}:
  \begin{align}
\omega P_{\lambda'}=    \widetilde{W}_\lambda. \label{eq:wisomegap}
  \end{align}
\end{remark}
  
\begin{proposition}\label{prop:prodofH}
  For each partition $\nu=(\nu_1,\ldots,\nu_k)$, 
  \begin{align*}
\prod_{i=1}^k    W_{(\nu_i)} =\sum_{\mu}(1-t)^{|\mu|-\mu_1}b_{\mu\nu}(t)W_{\mu}.
  \end{align*}
\end{proposition}

\begin{proof}
  It follows easily from the definition of plethysm that
  \begin{align*}
  \langle f,g\rangle=\langle f[X(1-t)], g\left[\frac{X}{1-t}\right]\rangle,
  \end{align*}
 for symmetric functions $f$ and $g$. By Proposition \ref{prop:bmunuascoeff}, 
  \begin{align*}
    b_{\mu\nu}(t)&=\frac{\prod_{i\geq 1}[\nu_i]_t!}{\prod_{i\geq 1}[\mu_i-\mu_{i+1}]_t!}\langle W_\mu,h_\nu\rangle \\
                 &=\frac{\prod_{i\geq 1}[\nu_i]_t!}{\prod_{i\geq 1}[\mu_i-\mu_{i+1}]_t!} \langle W_\mu[X(1-t)],h_\nu\left[\frac{X}{1-t}\right]\rangle \\
    &=(1-t)^{\mu_1}\prod_{i\geq 1}[\nu_i]_t!\, \langle \widetilde{W}_\mu,\prod_{i=1}^kh_{\nu_i}\left[\frac{X}{1-t}\right]\rangle ,
  \end{align*}
 where the last equality follows from Equation \eqref{eq:dualwhit}. For each positive integer $k$, we have $h_k=\widetilde{W}_{(k)}$ (follows from Bergeron \cite[Eq. 3.4]{bergeron2020survey}). Therefore
  \begin{align*}
    b_{\mu\nu}(t)&=(1-t)^{\mu_1}\prod_{i\geq 1}[\nu_i]_t! \langle \widetilde{W}_\mu,\prod_{i=1}^k\widetilde{W}_{(\nu_i)}\left[\frac{X}{1-t}\right]\rangle \\
    &=(1-t)^{\mu_1-|\mu|} \langle \widetilde{W}_\mu,\textstyle\prod_{i=1}^kW_{(\nu_i)}\rangle,
  \end{align*}
  again by Equation \eqref{eq:dualwhit}. It follows that $(1-t)^{|\mu|-\mu_1}b_{\mu\nu}(t)$ is the coefficient of $W_\mu$ in the $q$-Whittaker expansion of $\prod_{i=1}^kW_{(\nu_i)}$.
\end{proof}

\section{Full profiles}\label{sec:fullprofiles} 
In this section we obtain an explicit formula for $\sigma(\mu,\Delta)$ when $\mu$ is a partition of $n$. Arbitrary profiles are considered in Section \ref{sec:allprofiles}. 

Let $a_{\mu\lambda}(t)$ and $\tilde{a}_{ \mu\lambda}(t)$ denote the transition coefficients between the Hall-Littlewood functions and the elementary symmetric functions:
\begin{align*}
e_{\mu}=\sum_{\lambda}a_{\mu\lambda}(t)P_{\lambda}(x;t) \mbox{ and }  P_{\mu}(x;t)=\sum_{\lambda}\tilde{a}_{\mu\lambda}(t)e_{\lambda}.
\end{align*}
It is known that $a_{\mu\lambda}(1)$ is equal to the number of $(0,1)$-matrices with row sums $\mu$ and column sums $\lambda$. In fact, one can view the polynomials $a_{\mu\lambda}(t)$ as the generating polynomial for a suitably defined combinatorial statistic on $(0,1)$-matrices (Macdonald \cite[p. 211]{MR1354144}). On the other hand, the polynomials $\tilde{a}_{\mu\lambda}(t)$ do not have nonnegative coefficients in general and do not appear to have a nice combinatorial description. A very intricate explicit formula for $\tilde{a}_{\mu\lambda}(t)$ was found by Lassalle and Schlosser \cite[Thm. 7.5]{MR2222355} who gave an expression for the more general two-parameter coefficients in the elementary expansion of Macdonald polynomials. Recall that the Kostka numbers $K_{\lambda \mu}$ arise as coefficients in the monomial expansion of the Schur function, $ s_\lambda=\sum_{\mu} K_{\lambda\mu}m_{\mu}.$ The following lemma (Kirillov \cite[Eq. 0.3]{MR1768934}) expresses the polynomials $a_{\mu\lambda}(t)$ above in terms of Kostka numbers and Kostka-Foulkes polynomials.

\begin{lemma}\label{lem:dinkostka} 
  We have $a_{\mu\lambda}(t)=\sum_{\eta}K_{\eta'\mu}K_{\eta\lambda}(t).$
\end{lemma}

  For each partition $\lambda$, let $\epsilon_\lambda=(-1)^{|\lambda|-\ell(\lambda)}$, where $\ell(\lambda)$ denotes the number of parts of $\lambda$. The following theorem is the main result of this section.

\begin{theorem}\label{thm:fullprofiles}
  
For each partition $\mu$ of $n$,
\begin{align*}
  \sigma(\mu,\Delta)=\epsilon_{\mu'}\, q^{\sum_{j\geq 2}{\mu_j \choose 2}}\sum_{\lambda \vdash n}\tilde{a}_{\mu' \lambda}(q)X_\lambda(\Delta),
\end{align*}
for every prime power $q$ and every matrix $\Delta\in \MM_n(\Fq)$.
\end{theorem}
\begin{proof}
  Define
  \begin{align}
    \tilde{\sigma}(\mu,\Delta):=\epsilon_{\mu}q^{-\sum_{j\geq 2}{\mu'_j \choose 2}}\sigma(\mu',\Delta),\label{eq:sigtilda}
  \end{align} 
and note that the theorem is equivalent to $\tilde{\sigma}(\mu,\Delta)=\sum_{\lambda}\tilde{a}_{\mu \lambda}(q)X_\lambda(\Delta)$. To prove this, it suffices to show that
  \begin{align*}
X_\lambda(\Delta)=  \sum_{\mu}a_{\lambda \mu}(q) \tilde{\sigma}(\mu,\Delta),
  \end{align*}
  since the matrices $(\tilde{a}_{\lambda \mu}(q))$ and $(a_{\lambda \mu}(q))$ for $\lambda$ and $\mu$ varying over partitions of $n$ are mutually inverse to each other. In view of Theorem \ref{thm:formofanswer}, it suffices to show that Theorem \ref{thm:fullprofiles} holds for sufficiently large prime powers $q$. Fix a prime power $q\geq n$ and, for each partition $\nu$ of $n$, let $\Delta_{\nu}$ be a diagonal matrix of type $\nu$ over $\Fq$. Given a partition $\lambda$ of $n$, consider the system of equations
    \begin{align}
X_\lambda(\Delta_\nu)=  \sum_{\mu}x_{\mu}\, \tilde{\sigma}(\mu,\Delta_\nu), \label{eq:system}
  \end{align}
  in the variables $x_{\mu}$ as $\nu$ varies over all partitions of $n$. We claim that the system has a unique solution given by $x_{\mu}=a_{\lambda \mu}(q)$. To prove the claim, it suffices to show that the determinant $(\tilde{\sigma}(\mu,\Delta_\nu))$  for $\mu$ and $\nu$ varying over partitions of $n$ is nonsingular and Equation \eqref{eq:system} holds for $x_{ \mu}=a_{\lambda \mu}(q)$.
    
    Consider the determinant $D=\det(\tilde{\sigma}(\mu,\Delta_\nu))$ whose rows and columns are indexed by partitions $\mu,\nu$ of $n$ respectively, in lexicographic order (so that $(1^n)$ comes first and $(n)$ comes last). By Theorem~\ref{thm:diagonal} and the defining equation \eqref{eq:sigtilda} for $\tilde{\sigma}(\mu,\Delta)$, it follows that $D$ is equal to the determinant $\det (b_{\mu'\nu}(q))$ up to sign and a power of $(q-1)$. By permuting the rows of the determinant, it is clear that $\det (b_{\mu'\nu}(q))$ equals $\det (b_{\mu\nu}(q))$ up to sign.  By Remark \ref{rem:bvanishes}, $b_{\mu\nu}(q)=0$ unless $\mu\geq \nu$ in the dominance order and $b_{\mu\mu}(q)\neq 0$ for each prime power $q$. Thus the matrix $(b_{\mu\nu}(q))$ is lower triangular with nonzero diagonal entries. Therefore it is nonsingular, implying that $D$ is nonsingular. 
    
We use monomial generating functions to prove that Equation \eqref{eq:system} holds for $x_{ \mu}=a_{\lambda \mu}(q)$. Accordingly, let $m_\lambda$ denote the monomial symmetric function indexed by the partition $\lambda$. The monomial generating function of the left hand side of Equation \eqref{eq:system} is given by
  \begin{align*}
    \sum_{\lambda}X_\lambda(\Delta_\nu)m_{\lambda}=\prod_{i\geq 1}W_{(\nu_i)}(x;q),
  \end{align*}
by Example \ref{eg:diagonalifgf}.
  On the other hand, the monomial generating function for the right hand side of Equation \eqref{eq:system} with $x_{\mu}=a_{\lambda\mu}(q)$ is given by
  \begin{align*}
    \sum_{\lambda}\sum_{\mu}a_{\lambda \mu}(q) \tilde{\sigma}(\mu,\Delta_\nu)m_\lambda &=\sum_{\mu}\tilde{\sigma}(\mu,\Delta_\nu)\sum_{\lambda}a_{\lambda \mu}(q) m_\lambda.
  \end{align*}
  By Lemma \ref{lem:dinkostka}, the last expression above is equal to   
    \begin{align*}
      \sum_{\mu}\tilde{\sigma}(\mu,\Delta_\nu)\sum_{\lambda}\sum_{\eta}K_{\eta\lambda}K_{\eta'\mu}(q)m_\lambda&=\sum_{\mu}\tilde{\sigma}(\mu,\Delta_\nu)\sum_{\eta}K_{\eta'\mu}(q)s_\eta\\
                                                                                                              &=\sum_{\mu}\tilde{\sigma}(\mu,\Delta_\nu)W_{\mu'}(x;q),
\end{align*}                                                                                        where the last equality follows from Equation \eqref{eq:whittoschur}. By Theorem \ref{thm:diagonal}, the last sum above is equal to                        
\begin{align*}                                                                                            \sum_{\mu} \epsilon_\mu (q-1)^{\sum_{j\geq 2}\mu'_j}b_{\mu'\nu}(q) W_{\mu'}(x;q) &=\sum_{\mu} (1-q)^{|\mu|-\mu'_1}b_{\mu'\nu}(q) W_{\mu'}(x;q)\\
  &= \prod_{i=1}^k W_{(\nu_i)}(x;q),
    \end{align*}
    by Proposition \ref{prop:prodofH}. This completes the proof.
\end{proof}    
The theorem can be restated in the following alternate form using the Hall scalar product.

\begin{theorem}\label{thm:fullprofiles1}
  
For each partition $\mu$ of $n$,
\begin{align*}
  \sigma(\mu,\Delta) &=\epsilon_{\mu'}\, q^{\sum_{j\geq 2}{\mu_j \choose 2}}\langle  F_\Delta , \widetilde{W}_{\mu}(x;q)\rangle,
\end{align*}
for every prime power $q$ and every matrix $\Delta \in \MM_n(\Fq)$.
\end{theorem}
\begin{proof}
Follows by writing Theorem \ref{thm:fullprofiles} as $\sigma(\mu,\Delta)=\epsilon_{\mu'}\, q^{\sum_{j\geq 2}{\mu_j \choose 2}}\langle  F_\Delta , \omega P_{\mu'}\rangle$ and using Equation \eqref{eq:wisomegap}.
\end{proof}

   Theorem \ref{thm:fullprofiles1} gives the following combinatorial interpretation for the coefficients in the $q$-Whittaker expansion of the invariant flag generating function  $F_{\tau}(x;t)$ (see Equation \eqref{eq:ftau}) associated to a similarity class type $\tau$.

\begin{corollary}\label{cor:comb}
For each similarity class type $\tau=\{(d_1,\lambda^1),(d_2,\lambda^2),\ldots,(d_k,\lambda^k)\}$ of size $n$,
  \begin{align*}
 \prod_{i=1}^k p_{d_i}[\tilde{H}_{\lambda^i}(x;t)]=\sum_{\mu\vdash n}\epsilon_{\mu'}t^{-\sum_{j\geq 2}{\mu_j \choose 2}}\sigma(\mu,\tau) W_\mu(x;t), 
  \end{align*}
 where $\sigma(\mu,\tau)$ denotes the number of subspaces which have profile $\mu$ with respect to a matrix of similarity class type $\tau$ over the finite field $\mathbb{F}_t$ for sufficiently large prime powers $t$.
\end{corollary}
In view of the examples considered in Section \ref{sec:fdelta}, Corollary \ref{cor:comb} gives a combinatorial finite-field interpretation for the coefficients in the $q$-Whittaker expansion of the power sum symmetric functions, the complete homogeneous symmetric functions, products of modified Hall-Littlewood polynomials, and products of single-part $q$-Whittaker functions.

\begin{example}
  Specializing Corollary \ref{cor:comb} to the regular nilpotent and simple similarity class types and using Equations \eqref{eq:regnilprofiles} and \eqref{eq:simpleprofiles}, we obtain the following expansions for homogeneous and power sum symmetric functions:
  \begin{align*}
    h_n&=\sum_{\mu \vdash n}(-1)^{n-\mu_1}\left(\prod_{i\geq 2}t^{\mu_i+1 \choose 2}{\mu_{i-1} \brack \mu_i}_t\right) W_\mu(x;t),\\
  p_n&=\sum_{\mu \vdash n}(-1)^{n-\mu_1}\frac{t^n-1}{t^{\mu_1}-1}\left(\prod_{i\geq 2}t^{\mu_i \choose 2}{\mu_{i-1} \brack \mu_i}_t\right) W_\mu(x;t).
  \end{align*}
\end{example}
The $q$-Whittaker functions have recently arisen in the work of Karp and Thomas \cite{karp2023} in the context of counting certain partial flags compatible with a nilpotent endomorphism over a finite field. They also obtain an elegant probabilistic bijection between nonnegative integer matrices and pairs of semistandard tableaux. It would be interesting to relate their work to the counting problem considered in this paper. 

\section{Arbitrary profiles}\label{sec:allprofiles}
An explicit formula for $\sigma(\mu,\Delta)$ when $\mu$ is a partition of the ambient vector space dimension was obtained in Theorem \ref{thm:fullprofiles}. In this section we extend the result to arbitrary partitions $\mu$.

\begin{lemma}\label{lem:symfunction}
For each partition $|\mu|\leq n$, there exists a unique homogeneous symmetric function $G_{\mu}(x;t)\in \Lambda_{\QQ(t)}$ of degree $n$, such that
\begin{align*}
 \sigma(\mu,\Delta)=\langle F_\Delta,G_{\mu}(x;q) \rangle,
\end{align*}
for every prime power $q$ and each matrix $\Delta\in \MM_n(\Fq)$.
\end{lemma} 
\begin{proof}
   By Theorem \ref{thm:formofanswer}, there exist polynomials $g_{\mu\lambda}(t)\in \ZZ[t]$, depending only on $\mu$ and $\lambda$, such that
\begin{align*}
  \sigma(\mu,\Delta)=\sum_{\lambda \vdash n}g_{\mu \lambda}(q)X_\lambda(\Delta),
\end{align*}
for each prime power $q$ and each matrix $\Delta\in \MM_n(\Fq)$. If we define $ G_{\mu}(x;t):=\sum_{\lambda \vdash n}g_{\mu\lambda}(t)h_\lambda,$ then it is easily seen that $G_{\mu}(x;t)$ has the desired property, proving existence. To show uniqueness, suppose $G_{\mu}(x;t)$ and $\tilde{G}_{\mu}(x;t)$ are homogeneous symmetric functions of degree $n$ satisfying
\begin{align*}
\langle F_\Delta, G_{\mu}(x;q)\rangle=\langle F_\Delta, \tilde{G}_{\mu}(x;q)\rangle,
\end{align*}
for every prime power $q$ and every matrix $\Delta\in \MM_n(\Fq)$. Fix $q$ and allow $\Delta$ to vary over nilpotent matrices in $\MM_n(\Fq)$, to obtain
\begin{align*}
 \langle \tilde{H}_\lambda(x;q), G_{\mu}(x;q)-\tilde{G}_{\mu}(x;q)\rangle=0,
\end{align*}
for each partition $\lambda$ of $n$. The relations $ \tilde{H}_\lambda(x;q)=\sum_{\mu}\tilde{K}_{\mu\lambda}(q)s_\mu,$ are upper triangular with the diagonal entries $\tilde{K}_{\mu\mu}(q)$ being powers of $q$. Therefore, the $\tilde{H}_\lambda(x;q)$ form a basis for homogeneous symmetric functions in $\Lambda_{\QQ(t)}$ of degree $n$. It follows that $G_{\mu}(x;q)=\tilde{G}_{\mu}(x;q)$. Since this holds for each prime power $q$, it follows that $G_{\mu}(x;t)=\tilde{G}_\mu(x;t)$, proving uniqueness.
\end{proof}

The following result of Bender, Coley, Robbins and Rumsey was derived by a probabilistic argument and shows that $\sigma(\mu,\Delta)$ satisfies a remarkably simple system of equations.
\begin{theorem}\label{thm:bendereqns}\cite[Eq. 4]{MR1141317}
Let $n$ be a positive integer and suppose $\nu$ is a partition with $|\nu|<n$.  For each matrix $\Delta\in \MM_n(\Fq),$
  \begin{align*}
    \sum_{\mu:|\mu|\leq n} (-1)^{\mu_1}q^{-\mu \cdot \nu+{\mu_1 \choose 2}}  \sigma(\mu,\Delta)=0,
  \end{align*}
  where the sum is taken over all partitions $\mu$ of size at most $n$ (including the empty partition) and $\mu\cdot \nu:=\sum_{j\geq 1}\mu_j\nu_j.$
\end{theorem}

Bender, Coley, Robbins and Rumsey used Theorem \ref{thm:bendereqns} to derive an explicit formula for $\sigma(\mu,\Delta)$ in the cases where $\Delta$ is simple or regular nilpotent. Our approach is to use Theorem~\ref{thm:bendereqns} to solve for $\sigma(\mu,\Delta)$ for partitions satisfying $|\mu|<n$ using the formula already obtained for $\sigma(\mu,\Delta)$ in Theorem \ref{thm:fullprofiles} when $|\mu|=n$. The following theorem is our main result.

\begin{theorem}\label{thm:allprofiles}
For each partition $\mu$, 
\begin{align*}
 \sigma(\mu,\Delta)=\epsilon_{\mu'}\, q^{\sum_{j\geq 2}{\mu_j \choose 2}}\langle  F_\Delta,\widetilde{W}_{\mu}(x;q)\, h_{n-|\mu|} \rangle,
\end{align*}
for each prime power $q$ and each matrix $\Delta\in \MM_n(\Fq)$.
\end{theorem}

\begin{proof}
  By Lemma \ref{lem:symfunction}, there is a unique symmetric function $C_{\mu}(x;t)$ of degree $n$ such that
  \begin{align}
     \sigma(\mu,\Delta)=(-1)^{\sum_{j\geq 2}\mu_j}q^{\sum_{j\geq 2}{\mu_j \choose 2}}\langle  F_\Delta,C_\mu(x;q) \rangle,\label{eq:sigform}
  \end{align}
for each prime power $q$ and each matrix $\Delta\in \MM_n(\Fq)$. We wish to show that $C_\mu(x;t)=\widetilde{W}_{\mu}(x;t)h_{n-|\mu|}$, where $\widetilde{W}_\lambda$ denotes the dual $q$-Whittaker function indexed by $\lambda$. Consider the coefficients $c_{\eta\mu}(t)$ defined by
  \begin{align}
    C_{\mu}(x;t)=\sum_{\eta\vdash n}c_{\eta\mu}(t)\widetilde{W}_{\eta}(x;t), \label{eq:coeffs}
  \end{align}
   When $\mu$ is a partition of $n$, it follows from Theorem \ref{thm:fullprofiles1} that $C_{\mu}=\widetilde{W}_{\mu}$; therefore
  \begin{align}
    c_{\eta\mu}(t)=\delta_{\eta\mu} \quad (|\mu|=n),\label{eq:cetamu}
  \end{align}
  where $\delta$ denotes the Kronecker delta. Substituting the expression \eqref{eq:sigform} for $\sigma(\mu,\Delta)$ into the equations in Theorem \ref{thm:bendereqns}, we obtain 
  \begin{align*}
        \sum_{\mu:|\mu|\leq n} (-1)^{|\mu|}q^{-\mu \cdot \nu+\sum_{j\geq 1}{\mu_j \choose 2}} \langle  F_\Delta,C_\mu(x;q)\rangle=0, 
  \end{align*}
  for each partition $|\nu|<n.$ Since the above equations hold for each prime power $q$ and each matrix $\Delta\in \MM_n(\Fq)$, it is easily seen that 
  \begin{align*}
      \sum_{\mu:|\mu|\leq n} (-1)^{|\mu|}t^{-\mu \cdot \nu+\sum_{j\geq 1}{\mu_j \choose 2}} C_\mu(x;t)=0 \quad (|\nu|<n).
  \end{align*}
By Equation \eqref{eq:coeffs}, the expression on the left hand side above can be written as
    \begin{align*}
      &\sum_{\mu:|\mu|\leq n} (-1)^{|\mu|}t^{-\mu \cdot \nu+\sum_{j\geq 1}{\mu_j \choose 2}} \sum_{\eta\vdash n}c_{\eta\mu}(t)\widetilde{W}_{\eta}(x;t)\\
      &\quad=\sum_{\eta\vdash n}\sum_{\mu:|\mu|\leq n} (-1)^{|\mu|}t^{-\mu \cdot \nu+\sum_{j\geq 1}{\mu_j \choose 2}} c_{\eta\mu}(t)\widetilde{W}_{\eta}(x;t).
  \end{align*}
  The linear independence of the dual $q$-Whittaker functions implies that
  \begin{align*}
    \sum_{\mu:|\mu|\leq n} (-1)^{|\mu|}t^{-\mu \cdot \nu+\sum_{j\geq 1}{\mu_j \choose 2}} c_{\eta\mu}(t)=0 \quad \mbox{ for }|\nu|<n, |\eta|=n.
  \end{align*}
 By Equation \eqref{eq:cetamu}, this is equivalent to
  \begin{align*}
    \sum_{\mu:|\mu|< n} (-1)^{|\mu|}t^{-\mu \cdot \nu+\sum_{j\geq 1}{\mu_j \choose 2}} c_{\eta\mu}(t)+ (-1)^{n}t^{-\eta \cdot \nu+\sum_{j\geq 1}{\mu_j \choose 2}} =0 \quad (|\nu|<n),
  \end{align*}
for each partition $\eta$ of $n$. The idea now is to solve the above system of equations for the $c_{\eta\mu}(t)$ with $|\mu|<n$. Note that the number of unknowns is equal to the number of equations. We claim that the system has a unique solution given by 
  \begin{align}
    c_{\eta\mu}(t)=\psi_{\eta/\mu}(t):=\prod_{i\geq 1}{\eta_i-\eta_{i+1}\brack \eta_i-\mu_i}_t.\label{eq:uniqsol}
  \end{align}
Observe that $\psi_{\eta/\mu}(t)$ vanishes unless $\eta/\mu$ is a horizontal strip. Moreover, it is easy to see that $\psi_{\eta/\mu}(t)=\delta_{\eta\mu}$ when $\mu$ is a partition of $n$, in accordance with Equation~\eqref{eq:cetamu}.  Therefore, to prove the claim, it suffices to show the following.
  \begin{enumerate}
  \item The determinant with entries $(-1)^{|\mu|}t^{-\mu \cdot \nu+\sum_{j\geq 1}{\mu_j \choose 2}}$ for $\mu,\nu$ varying over partitions of integers less than $n$ is nonzero.
  \item For each partition $\eta$ of $n$,
    \begin{align*}
          \sum_{\mu:|\mu|\leq n} (-1)^{|\mu|}t^{-\mu \cdot \nu+\sum_{j\geq 1}{\mu_j \choose 2}} \psi_{\eta/\mu}(t)=0,
    \end{align*}
    for each partition $\nu$ of size less than $n$.
  \end{enumerate}
  The last two statements above are proved in Corollaries \ref{cor:nonsingular} and \ref{cor:psisatisfies} respectively. From Equation \eqref{eq:coeffs}, we obtain
  \begin{align*}
    C_{\mu}(x;t)&=\sum_{\eta\vdash n}\psi_{\eta/\mu}(t)\widetilde{W}_{\eta}(x;t)=\widetilde{W}_{\mu}(x;t)h_{n-|\mu|},
  \end{align*}
where the last equality follows by writing the Pieri rule for Hall-Littlewood polynomials as $P_{\mu'}e_{n-|\mu|}=\sum_{\eta\vdash n}\psi_{\eta/\mu}(t)P_{\eta'}$ (Macdonald \cite[p. 340, Eq. 6.24]{MR1354144}) and applying $\omega$. The result now follows by Equation \eqref{eq:sigform}.
\end{proof}
We now prove the two claims made in the proof of Theorem \ref{thm:allprofiles}. We introduce some notation, following Buck, Coley and Robbins \cite{MR1226346}, to state the relevant results. Let $\mathcal{P}$ denote the set of all partitions of nonnegative integers. Given partitions $\lambda=(\lambda_1,\lambda_2,\ldots)$ and $\mu=(\mu_1,\mu_2,\ldots)$ in $\mathcal{P}$, write $\lambda\subseteq \mu$ to mean $\lambda_i\leq \mu_i$ for each $i\geq 1$. With respect to the ordering $\subseteq$, the set $\mathcal{P}$ forms a distributive lattice called Young's lattice, denoted $Y$ (Stanley \cite[Sec. 7.2]{MR1676282}). An order ideal $I$ in $Y$ is a subset $I\subseteq Y$ such that if $\mu\in I$ and $\lambda \subseteq \mu$ for some $\lambda \in Y$, then $\lambda \in I$. Given a partition $\lambda=(\lambda_1,\lambda_2,\ldots)$, the shift operator is defined by ${\rm sh}(\lambda):=(\lambda_2,\lambda_3,\ldots)$. To each pair $(i,j)$ of nonnegative integers, we associate an indeterminate $x_{i,j}$. Given a pair of partitions $\lambda$ and $\mu,$ associate a monomial by  
\begin{align*}
  x_\lambda^\mu=\prod_{k\geq 1}x_{k,\lambda_k}^{\mu_k}.
\end{align*}
For partitions $\lambda,\mu$, write $\lambda\prec \mu$ if $\lambda \subseteq \mu$ and the partitions $\lambda$ and $\mu$ differ in precisely one coordinate. If they differ in the $r$th coordinate, let $\delta x(\mu,\lambda)$ denote the difference $x_{r,\mu_r}-x_{r,\lambda_r}$. Buck, Coley and Robbins \cite[Thm.~2]{MR1226346} proved the following nice product formula for the determinant with entries $x^\mu_{\lambda}.$

\begin{lemma}\label{lem:vandermonde}
  If $I\subset Y$ is a finite order ideal, then
  \begin{align*}
    \det(x_\lambda^\mu)_{\lambda,\mu\in I}=\pm\prod_{\lambda \in I}x_\lambda^{{\rm sh}(\lambda)}\prod_{\substack{ \lambda,\mu\in I\\\lambda \prec \mu}} \delta x(\mu,\lambda). 
  \end{align*}
\end{lemma}

\begin{corollary}\label{cor:nonsingular}
 Let $I\subset Y$ denote the order ideal consisting of all partitions of integers less than $n$. The determinant with entries $(-1)^{|\mu|}t^{-\mu \cdot \nu+\sum_{j\geq 1}{\mu_j \choose 2}}$ for $\mu,\nu\in I$ is nonzero. 
\end{corollary}
\begin{proof}
It suffices to prove that the determinant $\det(t^{-\mu\cdot \nu})_{\mu,\nu\in I}$ is nonzero. If we set $x_{i,j}=t^{-j}$  in Lemma \ref{lem:vandermonde}, then $x_{\nu}^\mu=t^{-\mu\cdot \nu}$ which is nonzero. For $\mu \prec \nu$, the quantity $\delta x(\mu,\nu)$ is also nonzero, being a difference of distinct integer powers of $t$. The result follows from these observations.
\end{proof}

\begin{lemma}\label{lem:psisum}
For $\psi_{\eta/\mu}(t)$ defined in Equation \eqref{eq:uniqsol}, we have
  \begin{align*}
    &\sum_{\mu}(-1)^{|\mu|}t^{-\mu \cdot \nu+\sum_{j\geq 1}{\mu_j \choose 2}}\psi_{\eta/\mu}(t)\\
    &\qquad\qquad\qquad=(-1)^{|\eta|}(1-t)^{\eta_1}t^{-\eta\cdot \nu+\sum_{j \geq 1}{\eta_j \choose 2}}\prod_{i\geq 1}{\nu_i-\eta_{i+1} \brack \eta_i-\eta_{i+1}}_t[\eta_i-\eta_{i+1}]_t!,
  \end{align*}
  where the sum is taken over all partitions $\mu$ such that $\eta/\mu$ is a horizontal strip. 
\end{lemma}
\begin{proof}
  Let $s(\eta,\nu)$ denote the sum in the statement of the lemma. We proceed by induction on the length $\ell(\eta)$ of $\eta$. When $\eta$ is the empty partition, it is easily seen that both sides of the identity reduce to 1 and the result holds in this case. Now suppose the result is true for $\ell(\eta)=\ell-1$ and let $\eta$ be a partition with $\ell$ parts. By conditioning on the first part $k$ of $\mu$, write the sum $s(\eta,\nu)$ as
  \begin{align*}
    s(\eta,\nu)&=\sum_{\mu}(-1)^{|\mu|}t^{-\mu \cdot \nu+\sum_{j\geq 1}{\mu_j \choose 2}}\prod_{i\geq 1}{\eta_i-\eta_{i+1}\brack \eta_i-\mu_i}_t\\
    &=\sum_{k=\eta_2}^{\eta_1}(-1)^k t^{-k\nu_1+{k \choose 2}}{\eta_1-\eta_2\brack \eta_1-k}_t\sum_{\tilde{\mu}}(-1)^{|\tilde{\mu}|}t^{-\tilde{\mu} \cdot \tilde{\nu}+\sum_{j\geq 1}{\tilde{\mu}_j \choose 2}}\psi_{\tilde{\eta}/\tilde{\mu}},
  \end{align*}
  where $\tilde{\lambda}$ denotes the partition obtained by removing the first part of $\lambda$.   The inner sum above is equal to $s(\tilde{\eta},\tilde{\nu})$. Therefore,
  \begin{align*}
    s(\eta,\nu)=    s(\tilde{\eta},\tilde{\nu})\sum_{k=\eta_2}^{\eta_1}(-1)^k t^{-k\nu_1+{k \choose 2}}{\eta_1-\eta_2\brack \eta_1-k}_t.
  \end{align*}
  With $j=k-\eta_2$, we obtain
  \begin{align}
    s(\eta,\nu)&=(-1)^{\eta_2}s(\tilde{\eta},\tilde{\nu})\sum_{j=0}^{\eta_1-\eta_2}(-1)^j  t^{-(j+\eta_2)\nu_1+{j+\eta_2 \choose 2}}{\eta_1-\eta_2\brack j}_t \notag \\
    &=(-1)^{\eta_2}t^{-\eta_2\nu_1+{\eta_2\choose 2}}s(\tilde{\eta},\tilde{\nu})\sum_{j=0}^{\eta_1-\eta_2}(-1)^j  t^{j(\eta_2-\nu_1)+{j \choose 2}}{\eta_1-\eta_2\brack j}_t. \label{eq:setanu}
  \end{align}
By applying the $q$-binomial theorem \cite[Eq. 1.87]{Stanley2012}
  \begin{align*}
    (1-x)(1-tx)\cdots (1-t^{n-1}x)=\sum_{j=0}^n (-1)^jt^{j \choose 2} {n \brack k}_tx^k,
  \end{align*}
  with $n=\eta_1-\eta_2$ and $x=t^{-(\nu_1-\eta_2)}$, the sum in Equation \eqref{eq:setanu} above becomes
  \begin{align*}
    &\left(1-\frac{1}{t^{\nu_1-\eta_2}}\right)\left(1-\frac{1}{t^{\nu_1-\eta_2-1}}\right)\cdots \left(1-\frac{1}{t^{\nu_1-\eta_1+1}}\right)\\
    &\quad =\frac{(t-1)^{\eta_1-\eta_2}}{t^{(\nu_1-\eta_2)(\eta_1-\eta_2)-{\eta_1-\eta_2 \choose 2}}}{\nu_1-\eta_2 \brack \eta_1-\eta_2}_t[\eta_1-\eta_2]_t!.
  \end{align*}
  By the inductive hypothesis,
  \begin{align*}
    s(\tilde{\eta},\tilde{\nu})=(-1)^{|\tilde{\eta}|}(1-t)^{\eta_2}t^{-\tilde{\eta}\cdot \tilde{\nu}+\sum_{j \geq 2}{{\eta}_j \choose 2}}\prod_{i\geq 2}{\nu_i-{\eta}_{i+1} \brack{\eta}_i-{\eta}_{i+1}}_t[{\eta}_i-{\eta}_{i+1}]_t!.\label{eq:tildas}
  \end{align*}
By substituting the above expression for $s(\tilde{\eta},\tilde{\nu})$ into Equation \eqref{eq:setanu}, a straightforward calculation shows that
  \begin{align*}
    s(\eta,\nu)&=(-1)^{\eta_2}t^{-\eta_2\nu_1+{\eta_2\choose 2}}s(\tilde{\eta},\tilde{\nu})\frac{(t-1)^{\eta_1-\eta_2}}{t^{(\nu_1-\eta_2)(\eta_1-\eta_2)-{\eta_1-\eta_2 \choose 2}}}{\nu_1-\eta_2 \brack \eta_1-\eta_2}_t[\eta_1-\eta_2]_t!\\
           &=      (-1)^{|\eta|}(1-t)^{\eta_1}t^{-\eta\cdot \nu+\sum_{j \geq 1}{\eta_j \choose 2}}\prod_{i\geq 1}{\nu_i-\eta_{i+1} \brack \eta_i-\eta_{i+1}}_t[\eta_i-\eta_{i+1}]_t!,
  \end{align*}
completing the inductive step and the proof.
\end{proof}

\begin{corollary}\label{cor:psisatisfies}
 For partitions $\nu,\eta$ satisfying  $|\nu|<|\eta|=n$, we have
  $$
  \sum_{\mu:|\mu|\leq n}(-1)^{|\mu|}t^{-\mu \cdot \nu+\sum_{j\geq 1}{\mu_j \choose 2}}\psi_{\eta/\mu}(t)=0.  
  $$
\end{corollary}
\begin{proof}
  Follows from Lemma \ref{lem:psisum} since the product $\prod_{i\geq 1}{\nu_i-\eta_{i+1} \brack \eta_i-\eta_{i+1}}_t$ 
  vanishes unless $\nu_i\geq \eta_i$ for each $i\geq 1$.
\end{proof}
\section{Partial profiles and anti-invariant subspaces}\label{sec:partialprofiles}
\begin{definition}
Given a matrix $\Delta \in \MM_n(\Fq)$, a subspace $W\subseteq \Fq^n$ has partial $\Delta$-profile $\rho=(\rho_1,\rho_2,\ldots,\rho_r)$ if  
\begin{align*}
  \dim (W+\Delta W+\cdots +\Delta^{j-1}W)=\rho_1+\rho_2+\cdots+\rho_{j} \mbox{ for }1\leq j\leq r.
\end{align*}
\end{definition}
Let $\pi(\rho,\Delta)$ denote the number of subspaces with partial $\Delta$-profile $\rho$. 

\begin{example}
We have $\pi((m),\Delta)={n \brack m}_q$ for each $\Delta\in \MM_n(\Fq)$ and $m\geq 0$.
\end{example}
Note that $\pi((m,0),\Delta)$ is the number of $m$ dimensional $\Delta$-invariant subspaces which is distinct from $\pi((m),\Delta)$ in general. Therefore, we think of $\rho$ as a tuple rather than an integer partition. It is evident that the number of subspaces with $\Delta$-profile $\mu=(\mu_1,\ldots,\mu_k)$ equals the number of subspaces with partial $\Delta$-profile $(\mu_1,\ldots,\mu_k,0)$.  

\begin{definition}\cite[p. 1]{MR2013452}
  Given $\Delta\in \MM_n(\Fq)$ and a positive integer $t$, a subspace $W$ of $\Fq^n$ is said to be $t$-fold $\Delta$-anti-invariant if
  \begin{align*}
    \dim(W+\Delta W+\cdots+\Delta^t W)=(t+1)\dim W.
  \end{align*}
\end{definition}
Thus an $m$-dimensional $\Delta$-anti-invariant subspace is precisely one with partial $\Delta$-profile $(m^{t+1})$. Anti-invariant subspaces were originally defined with $t=1$ by Barría and Halmos~\cite{MR748946}, motivated by earlier work of Hadwin, Nordgren, Radjavi and Rosenthal \cite{MR537082} on the weak density of certain sets of operators on Banach spaces. Barría and Halmos determined the maximum possible dimension of an anti-invariant subspace. Another proof of their result was given by Sourour \cite{MR822138}. This result was extended to arbitrary $t$ by Knüppel and Nielsen \cite{MR2013452}.  

\begin{theorem}
  Given $\rho=(\rho_1,\ldots,\rho_r)$ with $\rho_r\neq 0$ and $\Delta\in \MM_n(\Fq)$, the number of subspaces with partial $\Delta$-profile $\rho$ is given by
  \begin{align*}
    \pi(\rho,\Delta)=\langle \omega F_\Delta, G_\rho\rangle,
  \end{align*}
  where
  \begin{align*}
    G_\rho=(-1)^{\sum_{j\geq 2}\rho_j}q^{\sum_{j\geq 2}{\rho_j\choose 2}}\sum_{\substack{\eta \vdash n\\ \ell(\eta)=r}}\psi_{\eta/\rho}(q)P_{\eta'}(x;q).
  \end{align*}
\end{theorem}
\begin{proof}
We have
  \begin{align*}
    \pi(\rho,\Delta)&=\sum_{\substack{\mu:|\mu|\leq n\\ \mu_i=\rho_i (1\leq i\leq r)}}\sigma(\mu,\Delta)\\
    &=\sum_{\substack{\mu:|\mu|\leq n\\ \mu_i=\rho_i (1\leq i\leq r)}}(-1)^{\sum_{j\geq 2}\mu_j}\, q^{\sum_{j\geq 2}{\mu_j \choose 2}}\langle \omega  F_\Delta,P_{\mu'}(x;q)\, e_{n-|\mu|} \rangle,
  \end{align*}
  by Theorem \ref{thm:allprofiles}. Therefore $\pi(\rho,\Delta)=\langle \omega F_\Delta,G_\rho\rangle,$ where
  \begin{align*}
    G_\rho=\sum_{\substack{\mu:|\mu|\leq n\\ \mu_i=\rho_i (1\leq i\leq r)}}(-1)^{\sum_{j\geq 2}\mu_j}\, q^{\sum_{j\geq 2}{\mu_j \choose 2}}P_{\mu'}(x;q)\, e_{n-|\mu|}.
  \end{align*}
  By the Pieri rule for Hall-Littlewood polynomials, we have $P_{\mu'}e_{n-|\mu|}=\sum_{\eta}\psi_{\eta/\mu}P_{\eta'}.$ It follows that
  \begin{align*}
    G_\rho&=\sum_{\substack{\mu:|\mu|\leq n\\ \mu_i=\rho_i (1\leq i\leq r)}}(-1)^{\sum_{j\geq 2}\mu_j}\, q^{\sum_{j\geq 2}{\mu_j \choose 2}} \sum_{\eta\vdash n}\psi_{\eta/\mu}(q)P_{\eta'}(x;q)\\
    &=\sum_{\eta\vdash n}\sum_{\substack{\mu:|\mu|\leq n\\ \mu_i=\rho_i (1\leq i\leq r)}}(-1)^{\sum_{j\geq 2}\mu_j}\, q^{\sum_{j\geq 2}{\mu_j \choose 2}} \psi_{\eta/\mu}(q) P_{\eta'}(x;q).
  \end{align*}
  For each partition $\lambda$, define $\hat{\lambda}=(\lambda_1,\ldots,\lambda_r)$ and $\bar{\lambda}=(\lambda_{r+1},\lambda_{r+2},\ldots).$ Then the coefficient of $P_{\eta'}(x;q)$ in the sum above equals
  \begin{align*}
    &\sum_{\substack{\mu:|\mu|\leq n\\ \mu_i=\rho_i (1\leq i\leq r)}}(-1)^{\sum_{j\geq 2}\mu_j}\, q^{\sum_{j\geq 2}{\mu_j \choose 2}} \psi_{\eta/\mu}(q) \\
    &\qquad=(-1)^{\sum_{j\geq 2}\rho_j}\, q^{\sum_{j\geq 2}{\rho_j \choose 2}}\psi_{\hat{\eta}/\rho}(q)    \sum_{\bar{\mu}:|\bar{\mu}|\leq n-|\rho|}(-1)^{\sum_{j\geq 1}\bar{\mu}_j}\, q^{\sum_{j\geq 1}{\bar{\mu}_j \choose 2}} \psi_{\bar{\eta}/\bar{\mu}}(q).
  \end{align*}
  If $\bar{\eta}$ is not the empty partition, then the sum in the last expression above vanishes by Corollary \ref{cor:psisatisfies} with $\nu=\emptyset$. On the other hand, if $\ell(\eta)=r$, then it is easily seen that the inner sum above equals 1. It follows that
  \begin{align*}
     G_\rho=(-1)^{\sum_{j\geq 2}\rho_j}q^{\sum_{j\geq 2}{\rho_j\choose 2}}\sum_{\substack{\eta \vdash n\\ \ell(\eta)=r}}\psi_{\eta/\rho}(q)P_{\eta'}(x;q),
  \end{align*}
  proving the result.
\end{proof}
\begin{corollary}
  Given $\Delta\in \MM_n(\Fq)$ and a positive integer $t$, the number of $t$-fold $\Delta$-anti-invariant subspaces of dimension $m$ is given by
  \begin{align*}
(-1)^{mt}q^{t{m \choose 2}} \langle \omega F_\Delta,P_{((t+1)^m,1^{n-m(t+1)})}(x;q) \rangle,
  \end{align*}
  for $m(t+1)\leq n$.
\end{corollary}
\begin{proof}
  For $\rho=(m^{t+1})$, there is a unique partition $\eta$ of $n$ with $t+1$ parts such that $\eta/\rho$ is a horizontal strip, namely $\eta=( n-mt,m^t)$. In this case it is straightforward to verify that $\psi_{\eta/\rho}(q)=1$ and the result follows.
\end{proof}
\section{Application to Krylov subspace methods} 
\label{sec:krylov}
Let $\Delta\in \MM_n(\Fq)$ and consider a subset $S=\{v_1,\ldots,v_k\}$ of column vectors in $\Fq^n$. The \emph{truncated Krylov subspace} of order $\ell$ generated by $S$ is defined by
\begin{align*}
  {\rm Kry}(\Delta,S,\ell):=\left\{\sum_{i=1}^kf_i(\Delta)v_i:f_i(x)\in \Fq[x] \mbox{ and }\deg f_i<\ell \right\}.
\end{align*}
Let $\psi_{k,\ell}(\Delta)$ denote the probability of selecting a $k$-tuple of vectors uniformly at random from $\Fq^n$ such that the truncated Krylov subspace of order $\ell$ spanned by them is all of $\Fq^n$. Thus
\begin{equation} \label{eq:defpsi}
  \psi_{k,\ell}(\Delta):=\frac{1}{q^{nk}}|\{(v_1,\ldots,v_k)\in (\Fq^n)^k : {\rm Kry}(\Delta,\{v_1,\ldots,v_k\};\ell)=\Fq^n\}|.
\end{equation} 
 Computing $\psi_{k,\ell}(\Delta)$ is essential for analyzing a class of algorithms that solve large, sparse linear systems over finite fields, commonly encountered in number theory and computer algebra (Watkins \cite{MR2383888}). These algorithms, collectively referred to as Krylov subspace methods, have origins that can be traced back to contributions by Lagrange, Euler, Gauss, Hilbert and von Neumann, among others (Liesen and Strakoš \cite[p.\ 8]{MR3024841}). For instance, the linear algebra step in the Number Field Sieve, a well-known algorithm for large integer factorization, depends on Krylov subspace methods (Lenstra, Lenstra, Manasse and Pollard \cite{MR1321219}). Another example is Wiedemann’s algorithm, which is employed to determine the minimal polynomials of large matrices over finite fields (Liesen and Strakoš \cite[p.\ 19]{MR3024841}).  The quantity $\psi_{k,\ell}(\Delta)$ plays a key role in evaluating the effectiveness of these algorithms, and determining bounds on this probability represents a challenging and crucial task in the field (Brent, Gao and Lauder \cite[p.\ 277]{MR1982139}). We give an explicit formula for this probability as a scalar product of the flag generating function $F_{\Delta}$ with a suitably defined symmetric function.

\begin{theorem}\label{thm:krylov}
 For each matrix $\Delta\in \MM_n(\Fq)$, we have $\psi_{k,\ell}(\Delta)=\langle  F_\Delta,G(n,k,\ell)\rangle,$ where
  \begin{align*}
G(n,k,\ell)=    q^{-nk}\sum_{\substack{\mu\vdash n \\ \ell(\mu)\leq \ell }} (-1)^{n-\mu_1}(q-1)^{\mu_1}q^{\sum_{j\geq 1}{\mu_j \choose 2}}{k \brack \mu_1}_q [\mu_1]_q! \widetilde{W}_{\mu}(x;q).
  \end{align*}
\end{theorem}
\begin{proof}
  The number of tuples $(v_1,\ldots,v_k)$ of vectors in $\Fq^n$ which span a fixed subspace of dimension $m$ is given by 
  \begin{align*}
    (q^k-1)(q^k-q)\cdots (q^k-q^{m-1})={k \brack m}_q(q-1)^mq^{m \choose 2}[m]_q!.
  \end{align*}
  If $W$ denotes the subspace spanned by $(v_1,\ldots,v_k)$, then ${\rm Kry}(\Delta,\{v_1,\ldots,v_k\},\ell)=\Fq^n$ precisely when the $\Delta$-profile of $W$ is a partition of $n$ with at most $\ell$ parts. It follows that
  \begin{align*}
    \psi_{k,\ell}(\Delta)&= q^{-nk}\sum_{m=0}^k {k \brack m}_q(q-1)^m q^{m \choose 2}[m]_q!\sum_{\substack{\mu\vdash n \\ \mu_1=m\\ \ell(\mu)\leq \ell }} \sigma(\mu,\Delta)\\
    &=q^{-nk}\sum_{m=0}^k {k \brack m}_q(q-1)^m q^{m \choose 2}[m]_q!\sum_{\substack{\mu\vdash n \\ \mu_1=m\\ \ell(\mu)\leq \ell }} \epsilon_{\mu'}q^{\sum_{j\geq 2}{\mu_j \choose 2}}\langle  F_\Delta,\widetilde{W}_{\mu}(x;q)\rangle,
  \end{align*}
by Theorem \ref{thm:fullprofiles1}. Therefore, $\psi_{k,\ell}(\Delta)=\langle F_{\Delta},G(n,k,\ell)\rangle,$ where
  \begin{align*}
    G(n,k,\ell)&=q^{-nk}\sum_{m=0}^k {k \brack m}_q(q-1)^m q^{m \choose 2}[m]_q!\sum_{\substack{\mu\vdash n \\ \mu_1=m\\ \ell(\mu)\leq \ell }} \epsilon_{\mu'}q^{\sum_{j\geq 2}{\mu_j \choose 2}}\widetilde{W}_{\mu}(x;q)\\
    &=q^{-nk}\sum_{\substack{\mu\vdash n \\ \ell(\mu)\leq \ell }} (-1)^{n-\mu_1}(q-1)^{\mu_1}q^{\sum_{j\geq 1}{\mu_j \choose 2}}{k \brack \mu_1}_q [\mu_1]_q! \widetilde{W}_{\mu}(x;q).\qedhere
  \end{align*}
\end{proof}
\section{Recent developments}
Denote by $[n]$ the set of the first $n$ positive integers. A Hessenberg function is a weakly increasing function $\mm:[n]\to [n]$ satisfying  $\mm(i)\geq i$ for each $ i\in [n]$. For a linear operator $\Delta$ on $\Fq^n,$ the Hessenberg variety is defined by
\begin{align*}
  \he(\mm,\Delta):= \{\text{complete flags } V_1\subseteq V_2\subseteq \cdots \subseteq V_n=\Fq^n: \Delta V_i\subseteq V_{\mm(i)}\text{ for } i\in [n]\}.
\end{align*}
Drawing on results in this paper, the following theorem was proved in joint work with Abreu and Nigro \cite{abreu2024}. 
\begin{theorem}
For each operator $\Delta$ on $\Fq^n$, the number of $\Fq$-rational points on the Hessenberg variety $\he(\mm,\Delta)$ is given by
  $$
|\he(\mm,\Delta)|=\langle F_\Delta, \omega X_{G(\mm)}(x;q)\rangle,
$$
where $X_{G(\mm)}(x;t)$ denotes the chromatic quasisymmetric function of the unit interval graph $G(\mm)$ associated to $\mm$.
\end{theorem}
This result entails an expression for the Poincaré polynomials of complex Hessenberg varieties involving modified Hall-Littlewood polynomials.  
  
\section{Acknowledgments}
I thank Per Alexandersson, Amritanshu Prasad and Michael Schlosser for some helpful discussions. This research was partially supported by a MATRICS grant MTR/2017/000794 awarded by the Science and Engineering Research Board and an Indo-Russian project DST/INT/RUS/RSF/P41/2021.  
\printbibliography  
\end{document}